\documentclass{article}
\usepackage{amsmath,amsthm}
\usepackage{amssymb}
\usepackage{xspace}
\usepackage{xcolor}
\usepackage{graphicx}
\usepackage{tikz}   
\usepackage{diagbox}

\newtheorem{remark}{Remark}
\newtheorem{example}{Example}

\newtheorem{definition}{Definition}
\newtheorem{proposition}{Proposition}
\newtheorem{lemma}{Lemma}
\newtheorem{theorem}{Theorem}
\newtheorem{corollary}{Corollary}

\newcommand{\R}{\mathbb{R}}

\newcommand{\eqcl}[1]{\langle #1 \rangle}
\newcommand{\Ra}{\mathcal{R}}
\newcommand{\Nu}{\mathcal{N}} 

\newcommand{\RWP}{$\Ra$-WP}

\title{Classification of ill-posedness for bounded linear operators in Banach spaces}
\author{Bernd Hofmann\footnotemark[1] \and Stefan Kindermann\footnotemark[2]}

\begin{document}

\maketitle

\renewcommand{\thefootnote}{\fnsymbol{footnote}}
\footnotetext[1]{Chemnitz University of Technology, Faculty of Mathematics, 09107 Chemnitz, Germany.\\ Email: hofmannb@mathematik.tu-chemnitz.de}
\footnotetext[2]{Industrial Mathematics Institute, Johannes Kepler University Linz, Alternbergergstraße~69, 4040 Linz, Austria. Email: kindermann@indmath.uni-linz.ac.at}
\newcounter{enumcount}
\newcounter{enumcountroman}
\renewcommand{\theenumcount}{(\alph{enumcount})}
\bibliographystyle{plain}

\begin{abstract}
In this article, concepts of well- and 
ill-posedness for linear operators 
in Hilbert and Banach spaces are discussed. While these concepts are well understood in  Hilbert spaces, 
this is not the case in Banach spaces, as there are 
several competing definitions,   related 
to the occurrence of  uncomplemented subspaces. 
We provide an overview of the various definitions
and, based on this, discuss the classification 
of type I and type II ill-posedness in Banach spaces. 
Furthermore, a discussion of borderline 
(hybrid) cases in this classification is given 
together with several example instances of operators. 
\end{abstract}

\bigskip

{\parindent0em {\bf MSC2020:}}
47A52, 47B01, 47B02, 65J20

\bigskip

{\parindent0em {\bf Keywords:}}
Linear ill-posed operator equation, injective bounded linear operator, ill-posedness types, compact and non-compact operators, strictly singular operators
\bigskip

\section{Introduction}

Of Hadamard's three requirements for the \emph{well-posedness} of a model problem, namely (i) existence, (ii) uniqueness, and (iii) stability of the solution (cf.~\cite{Hadamard23}), the third demand (iii) is of particular interest for the model of a linear operator equation 
\begin{equation}\label{eq:opeq}
A x =y  \quad (x \in X,\; y \in Y),
\end{equation}
formulated in \emph{infinite-dimensional} abstract spaces $X$ and $Y$ (Hilbert or Banach spaces) with some bounded linear forward operator $A: X \to Y$. If at least one of the three requirements fails, Hadamard devalues the problems as \emph{ill-posed}. However, for such model problems \eqref{eq:opeq},
the requirement (i), which means surjectivity of $A$, and the requirement (ii), which means injectvity of $A$,  play a less important role for the practical handling, because exact data $y$
belonging to the range $\mathcal{R}(A)$ of $A$ ensure solution existence, and in case of non-trivial null-spaces $\mathcal{N}(A)$ of $A$
uniquely determined elements like minimum-norm solutions lead to the required uniqueness. Moreover, developed mathematical techniques
cover the case of multiple solutions. On the other hand, the stability requirement (iii) is  indispensable for well-posedness, and its failure cannot be overcome by simple technical tricks. 

Here, instability in the context of the violation of Hadamard's requirement (iii) means that small changes in
the data (right-hand side of equation \eqref{eq:opeq}) can lead to arbitrarily large changes in the solution. For a stability concept in the case of multiple solutions,
we refer to Remark~\ref{rem:quasidistance} in Section~\ref{sec:alternative} below. For linear operator equations, 
stability and instability can be expressed by the fact that the inverses $A^{-1}: \mathcal{R}(A) \subset Y \to X$ for injective operators $A$ (and their generalizations for non-injective $A$---see Section~\ref{sec:Banach2}) are \emph{bounded} and \emph{unbounded} linear operators,
respectively.

The stability aspect has been the substantial reason for M.~Z.~Nashed (cf.~\cite{Nashed86}) to characterize well-posedness and ill-posedness of linear operator equations with focus only on the stability requirement (iii).
Justified by regularizing properties, Nashed has furthermore distinguished in \cite{Nashed86} the type~I of ill-posedness from the 
type~II of ill-posedness. First we present a simplified (naive) version of Nashed's definition as Definition~\ref{def:type}, which is completely satisfactory for $X$ and $Y$ infinite-dimensional \emph{Hilbert spaces} as well as for \emph{injective} bounded linear operators $A$ when $X$ or $Y$ or both are infinite-dimensional Banach spaces. Definition~\ref{def:general} in Section~\ref{sec:Banach2}  presents a
refinement to do justice to the general \emph{Banach space}, where the null-space $\mathcal{N}(A)$ is not necessarily \emph{complemented} in the Banach space $X$. This may occur when $X$ is not isomorphic to a Hilbert space. 
Note that we will write \emph{$\Ra$-well-posed} and \emph{$\Ra$-ill-posed} if we refer to Definition~\ref{def:type} in order to distinguish it from \emph{well-posed} and \emph{ill-posed} of the concept by Definition~\ref{def:general} below.

\begin{definition}[Naive well-posedness and ill-posedness characterization] \label{def:type} $\qquad$
Let $A: X \to Y$ be a bounded linear operator mapping between the 
infinite-dimensional Banach spaces $X$ and $Y$.  

Then the operator equation \eqref{eq:opeq} is called \emph{well-posed} (or \emph{$\Ra$-well-posed})
\[ \text{if the range $\mathcal{R}(A)$ of $A$ is a closed subset of $Y$.} \]
Consequently \eqref{eq:opeq} is called ill-posed (or \emph{$\Ra$-ill-posed})
\[ \text{if the range $\mathcal{R}(A)$ is not closed, i.e., $\mathcal{R}(A) \not= \overline{\mathcal{R}(A)}^{Y}$.} \]
In the ill-posed case, the equation (\ref{eq:opeq}) is called \emph{ill-posed of type~I}
(or \emph{$\Ra$-ill-posed of type~I})
\[ \text{ if the range $\mathcal{R}(A)$ contains an {\sl infinite-dimensional closed subspace},}\]
and it is called ill-posed of type~II (or \emph{$\Ra$-ill-posed of type~II}) otherwise.
\end{definition}




For $X$ and $Y$ being both \emph{Hilbert spaces}, we illustrate by Figure~\ref{fig:Hilbert} in Section~\ref{sec:Hilbert}, which is analog to \cite[Fig.~1]{KindHof24}, this case distinction with respect to well-posedness and the different types of ill-posedness for bounded linear operators $A$ in the Hilbert space setting. As can be seen in the sequel, the situation in the Banach space setting is a bit more complex, because in Banach spaces non-compact operators can occur  that also lead to ill-posendness of type~II. In Section~\ref{sec:Banach1}, we will explain this in detail by Figure~\ref{fig:Banach1}, which we recall from \cite{FHV15}. Both figures correspond to Definition~\ref{def:type}. Illustrated by Figure~\ref{fig:Banach2}, Section~\ref{sec:Banach2} outlines the general Banach space situation, where uncomplemented null-spaces may occur that find its place in the Definition~\ref{def:general} along the lines of Nashed's seminal publication \cite{Nashed86} for the characterization of well-posedness and ill-posedness in the general case.

We precede the concept of \emph{strictly singular} operators by Definition~\ref{def:sso}, which is needed in the following sections, and we also
give in this context a list of well-known assertions that characterize strict singularity. 

\begin{definition}[Strictly singular operator (cf.~\cite{GoldThorp63})] \label{def:sso}
Let $X$ and $Y$ be Banach spaces, and let $A$ be a bounded
linear operator mapping $X$ into $Y$. Then $A$ is said to be \emph{strictly singular} if,
given any infinite-dimensional subspace $M$ of $X$, $A$ restricted to $M$ is
not an isomorphism (i.e., linear homeomorphism).
\end{definition}

\begin{proposition}[\mbox{\cite{Kato58}}] \label{rem:Kato}
A bounded linear operator $A:X \to Y$ mapping between Banach spaces is strictly singular if the closed subspaces $Z$ of $X$, for which the restriction $A|_{Z}$ has a bounded inverse, are necessarily
finite dimensional.
\end{proposition}   
Evidently, all compact operators are strictly singular.
In Hilbert spaces, the converse is true as well:
\begin{proposition}[\mbox{\cite[Remark p.~287f]{Kato58}}]
Let $A: X\to Y$ be a bounded operator mapping between the Hilbert spaces $X$ and $Y$. 
Then $A$ is strictly singular if and only if $A$ is compact.
\end{proposition}

\begin{lemma} \label{lem:finite}
If the range $\Ra(A)$ of the bounded linear operator $A: X \to Y$ mapping between infinite-dimensional Banach spaces is finite dimensional, then $A$ is strictly singular and compact with closed range. 

Conversely, let $A$  be strictly singular and possess a closed range $\Ra(A)$.
Then $\Ra(A)$ is finite dimensional and $A$ is also compact whenever either
\begin{itemize}
\item $X$ and $Y$ are Hilbert spaces 
\item or $A$ is an injective map between Banach spaces $X$ and $Y$. 
\end{itemize}
\end{lemma}
\begin{proof} 
The first assertion follows 
by the definition of strictly 
singular. For the second assertion, consider a strictly 
singular operator with closed 
range. In both cases, if $X,Y$
are Hilbert spaces or 
$A$ is injective we may find
a continuous right inverse 
on $\Ra(A)$: In the Hilbert space case, the pseudo-inverse $A^\dagger$ and in the injective 
Banach space case, this follows 
from the open mapping theorem. 
A classical result of 
Kato~\cite{Kato58} states that 
the strictly singular operators 
form an ideal. Hence the 
composition of $A$ with one of 
the above inverses gives the 
identity on 
$\tilde{X} = \Nu(A)^\bot$ in the Hilbert case or 
 $\tilde{X}= X$  in the injective 
 Banach space case, which must 
 be strictly singular then. 
 It 
 is well-known \cite{Kato58}
 that this can only be the 
 case if $\tilde{X}$
 is finite dimensional, which implies that 
 the range is so as well. 
\end{proof}

\begin{proposition}[\mbox{\cite{GoldThorp63}}]
\label{bpq}
Every bounded linear operator $A: \ell^p \to \ell^q$ with $p,q \in [1,\infty)$ and 
$p \not = q$ is strictly singular.
\end{proposition}
In particular, any \emph{embedding operator} $\ell^p \to \ell^q $
with $p<q$ is strictly singular and moreover \emph{non-compact}. 

\smallskip

For the next two propositions, see \cite[Proposition 2.C.3]{LiTz}. 
\begin{proposition}[Pitt's theorem]\label{th:pitt}
Let $p,r \in [1,\infty)$  with $r >p$. 
Then every bounded linear operator from $\ell^r \to \ell^p$ 
is compact. Also any bounded linear map $A:c _0 \to \ell^p$ with $p \in [1,\infty)$ is compact.  
\end{proposition}

\begin{proposition}
Let $p \in [1,\infty)$. 
Then a bounded linear map $\ell^p\to\ell^p$
is strictly singular if and only if it is
compact. 
\end{proposition} 

The next result from \cite[Corollary III.3.6]{GoldThorp63} 
extends some of the previous propositions. 
\begin{proposition}
 Let $X$ be $C(K)$ with $K\subset \R^n$  compact, or 
 $\ell^1$, $\ell^\infty$, $L^1(\Omega)$, or $L^\infty(\Omega)$ 
 or the dual of any of these spaces. 
 Then any bounded linear map $X \to Y$ with $Y$ a reflexive Banach space is 
 strictly singular. 
 In particular, this holds for 
 all bounded maps $L^1(\Omega)  \to L^p(\Omega) $ or 
 $L^\infty(\Omega) \to L^p(\Omega) $ with $p\in (1,\infty)$.
\end{proposition}

\smallskip

The sections of this work are organized as follows. In Section~\ref{sec:Hilbert}, the Hilbert space setting is discussed, and the specific situation there is illustrated by Figure~\ref{fig:Hilbert}. The Banach space setting for injective bounded linear operators will be outlined in Section~\ref{sec:Banach1} with respect to well- and ill-posedness, where now non-compact operators appear that are ill-posed of type~II in the sense of Definition~\ref{def:type}. The situation of this section is illustrated by Figure~\ref{fig:Banach1}. Section~\ref{sec:Banach2} plays a central role for this paper, because there for the general case of bounded linear operators in Banach spaces, an adapted Definition~\ref{def:general} is introduced and illustrated by Figure~\ref{fig:Banach2}.
This definition covering complemented and uncomplemented null-spaces highlights well-posedness and ill-posedness of type~I and type~II in the sense of Nashed with the occurring distinguished facets of this Banach space world. Section~\ref{sec:alternative} completes the paper by presenting and evaluating a series of alternative concepts of well-posedness from the literature. In Subsection~\ref{sec:sub1}, characterizations of well-posedness by generalized inverses are given, whereas in Subsection~\ref{sec:sub2} alternative concepts for types of ill-posedness are under consideration, illustrated by Figure~\ref{fig:FHV25} and a couple of enlightening  examples in Section~\ref{sec:alternative}.



\section{Hilbert space setting} \label{sec:Hilbert}

The diagram in Figure~1 below illustrates, under the auspices of Definition~\ref{def:type}, which means that well-posed and ill-posed is to understand in the sense of $\Ra$-well-posed and $\Ra$-ill-posed, the different cases that are possible for equation \eqref{eq:opeq} in the Hilbert space setting, where the bounded linear operator $A:X \to Y$ maps between the 
infinite-dimensional Hilbert spaces $X$ and $Y$.

\begin{figure}
\begin{center}
           \begin{tikzpicture}
            \draw[fill={rgb:black,0.5;white,10}, thick]  plot[smooth, tension=.7, fill] coordinates {(-3.5,0.5) (-3,2.5) (-1,3.5) (1.5,3.5)
            (4,3.5) (6,2.5) (6.,-0.1) (4.,-2) (0,-1.9) (-3,-2) (-3.5,0.5)} ;  
                \draw[fill={rgb:black,4;white,10}, thick]  plot[smooth, tension=.7, fill] coordinates {(1.5,3.5)
            (4,3.5) (6.5,2.5) (6.5,0.0) (4.5,-2) (0.8,-1.95)    (1.1,0.)  (1.75,1.7)  (1.5,3.5)} ;

            \begin{scope}[shift={(-1,1.5)}]
            \draw[fill={rgb:black,4;white,10},thick]  plot[smooth, tension=.7, fill] coordinates {(2,-0.7) (1.5, 0) (2,1) 
            (2.6 ,1.3)
            (2.7 ,1.)
            (2.8 ,0.3)
            (2.65 ,-0.3)
            (2.35,-0.75) 
            (2,-0.7)    
            } ;
            \end{scope}

            \draw[thick]
             (1.5,3.5) .. controls (1.8,1.8) and (2.1,2) .. (1.25,0.45) ;
             
            \draw[thick]
             (-3.4,1.5) .. controls (-2.2,1.3) and (-2.0,1.7) .. (1.25,0.45) ;             

 \draw[->,thick]        (1.5,4.0)   -- (1.5,2.5); 
  \node at (1.5,4.7) {\begin{tabular}{c} well-posed compact \\    
                          = finite dimensional range \\
                           (strictly singular with closed range)
                       \end{tabular}
};

\node at (-0.7,2.3) {\textbf{well-posed}};
\node at (-0.7,2.0) {(closed range)};
\node at (-1.5,0.2) {\textbf{ill-posed type I}};
\node at (-1.5,-0.2) {\textbf{non-compact}};
\node at (-1.5,-0.7) {(non-closed range)};
\node at (3.9,0.7) {\textbf{ill-posed type II}};
\node at (3.9,0.3) {\textbf{compact}};
\node at (3.9,-0.2) {(non-closed range)};
\node at (-0.2,1.1) [rotate=-15] {operator not};
\node at (-0.8,0.9) [rotate=-15] {strictly singular};
\node at (4.7,1.7) {operator strictly};
\node at (4.9,1.4) {singular};
\node at (1.0,1.7) [rotate=70]  {compact};
\node at (1.3,1.7) [rotate=70]  {well-posed};
\end{tikzpicture}        
\end{center}
\caption{Case distinction for bounded linear operators between 
infinite-dimensional \emph{Hilbert spaces}. Here, strictly singular operators are always compact. 
}\label{fig:Hilbert}
\end{figure}
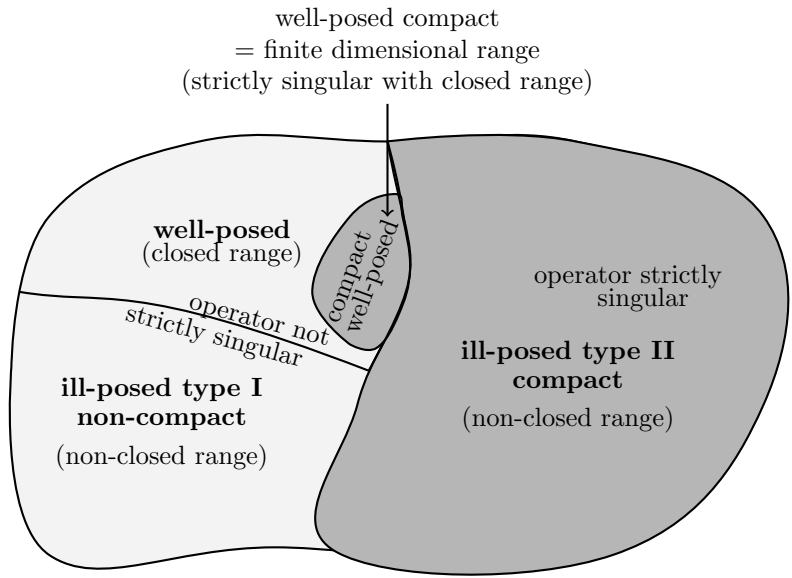

Compact and thus strictly singular operators with 
infinite-dimensional range possess a non-closed range and are therefore $\Ra$-ill-posed. This class of operators, which are characterized by an $\Ra$-ill-posedness of type~II, fills out the right part of Figure~\ref{fig:Hilbert}. Compact operators with finite-dimensional range are also strictly singular, but $\Ra$-well-posed and belong with its closed range (cf.~Lemma~\ref{lem:finite}) to the left part of Figure~\ref{fig:Hilbert}. The main part of the left half of this figure consists of not strictly singular operators, the range of which posssesses a closed infinite-dimensional subspace of $Y$. These operators are either $\Ra$-well-posed with closed range or alternatively $\Ra$-ill-posed of type~I with non-closed range. 
$\Ra$-well-posedness takes place there for continuously invertible operators with the identity operator as the simplest example. $\Ra$-ill-posed operators of type~I can, for example, be found in the class of bounded multiplication operators with essential zeros of the multiplier functions (cf.~\cite{nmh22,WernerHof25}) and by Cesaro and Hausdorff moment operators (cf.~\cite{Gerth21,KindHof24}).   

Let us emphasize that in Hilbert spaces under Definition~\ref{def:type} ill-posed situations are clearly distinguished by the fact that all ill-posed operators of type~II are compact, whereas all ill-posed operators of type~I are non-compact. The next section will show that this is not in general true in Banach spaces.

\section{Banach space setting for injective operators} \label{sec:Banach1}

The diagram in Figure~\ref{fig:Banach1} 
illustrates, under the auspices of Definition~\ref{def:type}, the different cases occurring for equation \eqref{eq:opeq} in the Banach space setting when the bounded linear operator
$A:X \to Y$ mapping between the 
infinite-dimensional Banach spaces $X$ and $Y$ is \emph{injective}. 
It follows from Definition~\ref{def:sso} that strictly singular operators $A$
cannot lead to $\Ra$-well-posedness or $\Ra$-ill-posendess of type~I for the operator equation \eqref{eq:opeq}, because for injective strictly singular operators the range $\mathcal{R}(A)$ of $A$ does not contain an infinite-dimensional closed subspace of $Y$. Vice versa, ``any bounded linear operator between Banach spaces whose range does not contain any
infinite-dimensional closed subspace is strictly singular''
(\cite[Remark]{GoldThorp63}). Consequently, the property of strict singularity separates the left part of Figure~\ref{fig:Banach1} from the right part. 

Note that here the area of compact operators is a proper subset of the area of strictly singular operators (right part of Figure~\ref{fig:Banach1}) that are always $\Ra$-ill-posed of type~II (see, e.g.,~\cite[Prop.~4.6]{FHV15}). On the one hand, for such injective bounded linear operators $A$, finite-dimensional ranges  $\Ra(A)$ do not occur. On the other hand, in Banach spaces, injective bounded non-compact and strictly singular linear operators exist (see Example~\ref{ex:embeddings} below concerning embedding operators $\mathcal{E}$ from $\ell^p$ to $\ell^q$ for $p<q$), which proves that the compact operators are not enough to fill out the right part of the figure. 
This indicates a substantial difference between Figure~\ref{fig:Hilbert} and Figure~\ref{fig:Banach1}, because compact operators with infinite-dimensional range completely fill out the right part of Figure~\ref{fig:Hilbert} in the Hilbert space setting.

Examples for compact operators in Banach spaces can be found by injective bounded operators $\mathcal{E}: \ell^p \to \ell^r$ with $r < p $ and $1 \le p,r <\infty$ due to
Pitt's theorem (cf.~Proposition~\ref{th:pitt}). Also injective linear Fredholm integral operators mapping in $C[0,1]$ with continuous non-degnerating kernel functions represent examples of compact operators with infinite-dimensional range.

Not strictly singular injective operators $A$ (left part of Figure~\ref{fig:Banach1}), the range of which possesses an infinite-dimensional closed subspace of $Y$,
lead either to \emph{$\Ra$-well-posed} operator equations \eqref{eq:opeq}  with closed range if $A$ is \emph{continuously invertible} or to equations \emph{$\Ra$-ill-posed of type~I} if $A$ has a non-closed range $\mathcal{R}(A) \not=\overline{\mathcal{R}(A)}^Y$.

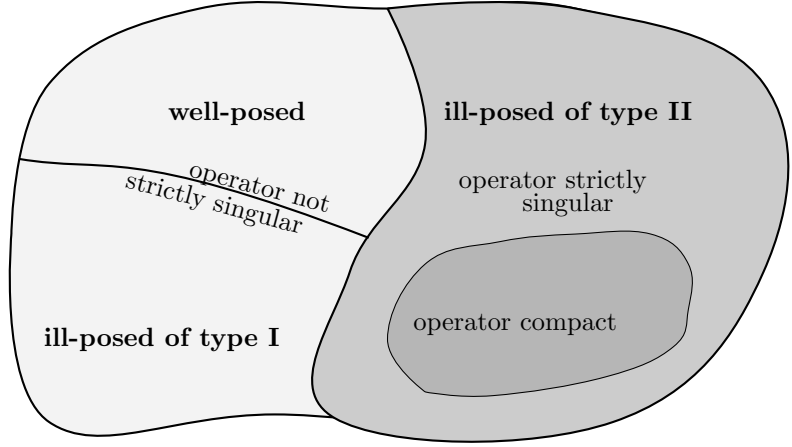
\begin{figure}[ht]
\begin{center}
           \begin{tikzpicture}
            \draw[fill={rgb:black,0.5;white,10}, thick]  plot[smooth, tension=.7, fill] coordinates {(-3.5,0.5) (-3,2.5) (-1,3.5) (1.5,3.5)
            (4,3.5) (6,2.5) (6.,-0.1) (4.,-2) (0,-1.9) (-3,-2) (-3.5,0.5)} ;  
                \draw[fill={rgb:black,2.5;white,10}, thick]  plot[smooth, tension=.7, fill] coordinates {(1.5,3.5)
            (4,3.5) (6.5,2.5) (6.5,0.0) (4.5,-2) (0.8,-1.95)    (1,0.)  (2.0,1.7)  (1.5,3.5)} ;

            \begin{scope}[shift={(0,-0.9)}]
            \draw[fill={rgb:black,4;white,10}]  plot[smooth, tension=.7, fill] coordinates {(2,-0.7) (1.5, 0) (2,1) (3 ,1.3)
            (4 ,1.4)       (5 ,1.4)       (5.5 ,1.0)   (5.5 ,0.4)     (5.3,-0.1)     (4.4,-0.5)     (2.9,-0.75)     (2,-0.7)    
            } ;
            \end{scope}
            \draw[thick]
             (-3.4,1.5) .. controls (-2.2,1.3) and (-2.0,1.7) .. (1.25,0.45) ;             
\node at (-0.5,2.1) {\textbf{well-posed}};
\node at (-1.5,-0.9) {\textbf{ill-posed of type I}};
\node at (3.9,2.1) {\textbf{ill-posed of type II}};
\node at (-0.2,1.1) [rotate=-15] {operator not};
\node at (-0.8,0.9) [rotate=-15] {strictly singular};
\node at (3.7,1.2) {operator strictly};
\node at (3.9,0.9) {singular};
\node at (3.2,-0.7) {operator compact};
\end{tikzpicture}        
\caption{Case distinction for \emph{injective} bounded linear operators mapping between
infinite-dimensional \emph{Banach spaces} (cf.~\cite[p.287]{FHV15}).}\label{fig:Banach1}
\end{center}
\end{figure}

\section{Banach space setting for general operators} \label{sec:Banach2}

The characterization of well-posedness and ill-posedness of operator equations \eqref{eq:opeq} with \emph{not necessarily injective} bounded linear operators $A: X \to Y$ mapping between 
infinite-dimensional Banach spaces $X$ and $Y$ in Nashed's sense has to take into account that the null-space $\mathcal{N}(A)$ of the operator $A$ need not be
(topologically) complemented in $X$. By considering the direct sum $X=\mathcal{N}(A)\oplus U$, the null-space $\mathcal{N}(A)$ is called \emph{complemented} in $X$ if $U$ is also a closed subspace of $X$ as the subspace $\mathcal{N}(A)$ is, otherwise $\mathcal{N}(A)$ is called \emph{uncomplemented} in $X$. Banach spaces $X$, which are not isomorphic  to a Hilbert space, always contain uncomplemented subspaces (see, e.g.,~\cite{LiTz}). Consequently, for the general Banach space setting, Definition~\ref{def:type} has to be replaced by the following Definition~\ref{def:general} which covers both the case of complemented null-space $\mathcal{N}(A)$ and the case that $\mathcal{N}(A)$ is uncomplemented in $X$.  

\begin{definition}[Well- and ill-posedness characterization] \label{def:general}
Let $A: X \to Y$ be a bounded linear operator mapping between the 
infinite-dimensional Banach spaces $X$ and $Y$.  

Then the operator equation \eqref{eq:opeq} is called \emph{well-posed} if 
\begin{align*} &\text{the range $\mathcal{R}(A)$ of $A$ is a closed subset of $Y$ and, moreover,} \\ 
&\text{the null-space $\mathcal{N}(A)$ is \emph{complemented} in $X$;}
\end{align*}
otherwise the equation (\ref{eq:opeq}) is called \emph{ill-posed}. 

In the ill-posed case, (\ref{eq:opeq}) is called \emph{ill-posed of type I}   if 
\begin{align*} 
&\text{the range $\mathcal{R}(A)$ contains an \emph{infinite-dimensional closed subspace} $M$ of $Y$ and} \\ 
&\text{the null-space $\mathcal{N}(A)$ is \emph{complemented} in $A^{-1}[M]$.}
\end{align*} 
Here, $A^{-1}[M]$ is  the inverse image of $M$ under $A$.

Otherwise the ill-posed
equation (\ref{eq:opeq}) is called \emph{ill-posed of type~II}.
\end{definition}

The following note is a consequence of the corresponding discussions in \cite[Section~1]{Flemmingbuch18}: For the linear bounded operator $A:X \to Y$ defined on the Banach space $X$ with a direct sum decomposition as $X=\mathcal{N}(A)\,\oplus \, U$, the restriction 
$A|_U: U \to \mathcal{R}(A)$ of $A$ is a bijective mapping, and we denote its inverse
$A^\dagger_U: \mathcal{R}(A) \to U$ as \emph{generalized pseudoinverse} of $A$. In the injective case, $A^\dagger_U$ and $A^{-1}:\mathcal{R}(A) \to X$ coincide.

\begin{proposition}[\mbox{\cite[Proof of Prop.~1.10 and Prop.~1.11]{Flemmingbuch18}}]
\label{pro:Flemming}
If, for $A: X \to Y$ mapping between infinite-dimensional Banach spaces $X$ and $Y$, the null-space $\mathcal{N}(A)$ is complemented in $X$, then the generalized pseudoinverse $A_U^\dagger$ is bounded on any
closed subspace of the range $\mathcal{R}(A)$. If $\mathcal{N}(A)$ is uncomplemented in $X$, then $A_U^\dagger$ is always an unbounded operator.
\end{proposition}

We emphasize that we can justify the central role of Definition~\ref{def:general} with respect to well-posedness concepts in this paper with the following corollary, which is an immediate consequence of Proposition~\ref{pro:Flemming}.  

\begin{corollary} \label{cor:pseudinverse}
For a bounded linear operator $A: X \to Y$ mapping between the 
infinite-dimensional Banach spaces $X$ and $Y$, the operator equation \eqref{eq:opeq} is well-posed in the sense of
Definition~\ref{def:general} if and only if there is a bounded linear generalized pseudoinverse $A^\dagger_U: \Ra(A) \to U$ for a decomposition $X=\Nu(A)\! \oplus\, U$ of the Banach space $X$ with null-space $\Nu(A)$.
\end{corollary}

As in Lemma~\ref{lem:finite}, 
the well-posed and strictly 
singular operators are those 
with finite-dimensional range.
\begin{lemma}
Let $A$ be  well posed and 
strictly singular. Then $\Ra(A)$ 
is finite dimensional and $A$ is also compact. 
Conversely, if $\Ra(A)$ 
is finite dimensional, then 
$A$ is well-posed and strictly
singular. 
\end{lemma}
\begin{proof}
If $A$ is well-posed, the 
generalized pseudoinverse 
$A_U^\dagger$ is bounded. 
Since the strictly singular operator 
form an ideal, we have that 
 $A_U^\dagger A$ is strictly singular, but this is 
the identity on $U$. Thus, 
$U$ is finite dimensional 
and hence $\Ra(A) = A(U)$
is finite dimensional as well.
Conversely, if the range 
is finite dimensional, 
the operator is compact 
and thus strictly singular 
and the range is closed. By the open 
mapping theorem, we have an isomorphism 
of $\Ra(A)$ to $X/\Nu(A)$. Thus, $X/\Nu(A)$
is finite dimensional, which means that 
$\Nu(A)$ has finite codimension. It is well-known that then $\Nu(A)$ is complemented
(cf., e.g.,  \cite[Proposition~11.6]{Brezis}), hence $A$ is well-posed. 
\end{proof}
%
\begin{proposition} \label{pro:iff}
Let $A:X \to Y$ be such that the operator equation
\eqref{eq:opeq} is ill-posed problem in the sense of Definition~\ref{def:general}.
Then the equation is ill-posed of type~II if and only if $A$ is 
strictly singular. 
\end{proposition}
\begin{proof}
If $A$ is not ill-posed of type~II, 
$\Ra(A)$ contains a closed
infinite-dimensional subspace 
$M = A[X_1]$, where $\Nu(A)$ is complemented in the subspace $X_1$ of $X$.
Then the complement $X_2$ is closed and $M = A[X_2]$.
Now $A|_{X_2}$ is an injective operator onto $M$, and 
by the open mapping theorem there exists a bounded 
inverse on $M$. Thus, $A$ is not strictly singular. 
If, conversely, $A$ is not strictly singular, then there exists 
an infinite-dimensional closed subspace $X_1$ of $X$ such that the restriction $A|_{X_1}$ 
of the operator $A$ with $M=A[X_1]$ has a
bounded inverse. 
Now $X_1 \cap \Nu(A|_{X_1}) = \{0\}$, 
hence $\Nu(A|_{X_1})$ is trivially complemented in $A^{-1}[M]$
and $M$ is a closed subspace of $Y$. Since $A$ was assumed to 
be ill-posed, this means that $A$ is not ill-posed of type~II. 
\end{proof}

Figure~\ref{fig:Banach2} reflects the refined situation of Definition~\ref{def:general}. At first glance, it seems that the differences between Figure~\ref{fig:Banach1} and Figure~\ref{fig:Banach2} are not really great. But that is only an appearance, because the strictly singular operators with infinite-dimensional range
on the right part of Figure~\ref{fig:Banach2} contain cases of uncomplemented null-spaces (see Proposition~\ref{pro:hyb} and Example~\ref{ex:hybrid} below).

\begin{figure}[ht]
\begin{center}
           \begin{tikzpicture}
            \draw[fill={rgb:black,0.5;white,10}, thick]  plot[smooth, tension=.7, fill] coordinates {(-3.5,0.5) (-3,2.5) (-1,3.5) (1.5,3.5)
            (4,3.5) (6,2.5) (6.,-0.1) (4.,-2) (0,-1.9) (-3,-2) (-3.5,0.5)} ;  
                \draw[fill={rgb:black,2.5;white,10}, thick]  plot[smooth, tension=.7, fill] coordinates {(1.5,3.5)
            (4,3.5) (6.5,2.5) (6.5,0.0) (4.5,-2) (0.8,-1.95)    (1.1,0.)  (1.75,1.7)  (1.5,3.5)} ;

            \begin{scope}[shift={(-1,1.5)}]
            \draw[fill={rgb:black,4;white,10},thick]  plot[smooth, tension=.7, fill] coordinates {(2,-0.7) (1.5, 0) (2,1) (3 ,1.3)
            (4 ,1.4)       (5 ,1.4)       (5.5 ,1.0)   (5.5 ,0.4)     (5.3,-0.1)     (4.4,-0.5)     (2.9,-0.75)     (2,-0.7)    
            } ;
            \end{scope}
            
                    \draw[fill={rgb:black,2.5;white,10}, thick]  plot[smooth, tension=.7, fill] coordinates
                    {
               (0.8,-1.75)  (0.55,-1.5)  (0.7,-1.)   (1.15,-0.3)    (2.,-0.5)    (2.,-1.5)    (0.8,-1.75)  
                    };

            \draw[thick]
             (1.5,3.5) .. controls (1.8,1.8) and (2.1,2) .. (1.25,0.45) ;
             
            \draw[thick]
             (-3.4,1.5) .. controls (-2.2,1.3) and (-2.0,1.7) .. (1.25,0.45) ;             

 \draw[->,thick]        (1.3,-3)   -- (1.3,-1.2);
 \draw[->,thick]        (1.5,4.5)   -- (1.5,1.7); 
  \node at (1.5,5.2) {\begin{tabular}{c} well-posed strictly singular\\ = well-posed compact \\    
                          = finite dimensional range
                       \end{tabular}
};
 \node at (1.3,-3.5) {{\begin{tabular}{c} strictly singular with $\Ra(A)$ containing\\ closed infinite dimensional subspace\end{tabular}}};

 \node at (1.4,-0.9) {hybrid};
\node at (-0.5,2.1) {\textbf{well-posed}};
\node at (-1.5,-0.9) {\textbf{ill-posed of type I}};
\node at (3.9,-.5) {\textbf{ill-posed of type II}};
\node at (-0.2,1.1) [rotate=-15] {operator not};
\node at (-0.8,0.9) [rotate=-15] {strictly singular};
\node at (4.7,0.7) {operator strictly};
\node at (4.9,0.4) {singular};
\node at (2.2,1.7) [rotate=10]  {operator compact};
\end{tikzpicture}        
\caption{Case distinction for bounded linear operators between
infinite-dimensional \emph{Banach spaces} with \emph{complemented} and \emph{uncomplemented} null-spaces.}\label{fig:Banach2}
\end{center}
\end{figure}
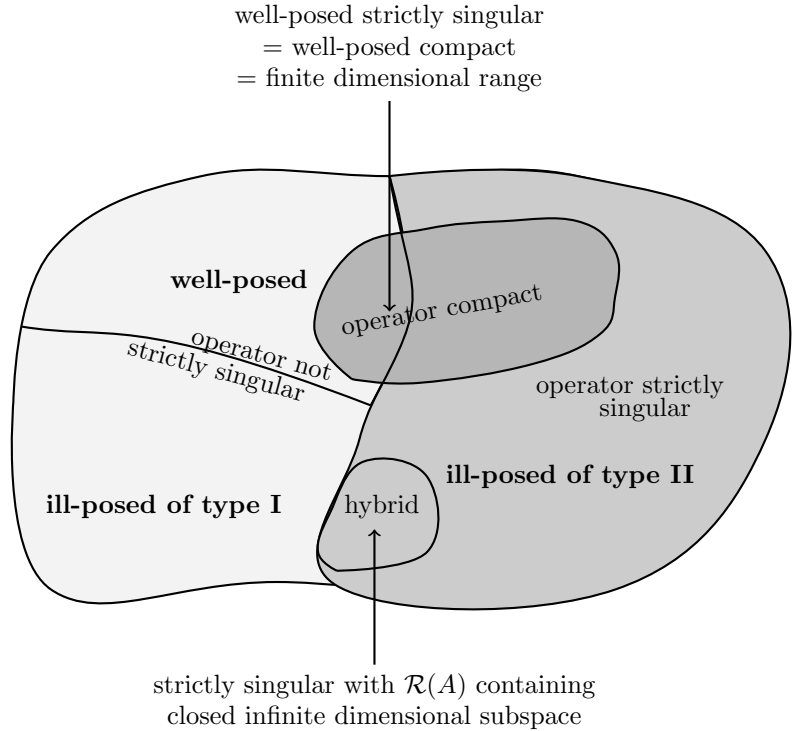

Beforehand we mention that uncomplemented null-spaces according to Definition~\ref{def:general} are not compatible with well-posed situations and ill-posed situations of type~I, which are both assigned to the left part of Figure~\ref{fig:Banach2}.
Since the well-posed case (compact or not) is well-understood, we focus with respect
to explanations of Figure~\ref{fig:Banach2} on the ill-posed operators, and especially on ill-posedness of type~II.

We highlight, in particular, the hybrid case incorporated in Figure~\ref{fig:Banach2} as small circle (see also Example~\ref{ex:hybrid}). 
\begin{definition} \label{def:hybrid}
In the setting of Definition~\ref{def:general}, let us call an operator $A$ of \emph{hybrid} type when $A$ is strictly singular and  its range $\mathcal{R}(A)$ contains an 
infinite-dimensional closed subspace of $Y$. 
\end{definition}

We label such cases as hybrid since 
their type  is  classified 
differently when switching between
Definition~\ref{def:type} and 
Definition~\ref{def:general}:
They are type II by the latter 
definition but type I (or well-posed) by the former.

Next we verify that hybrid operators, which are ill-posed of type~II, 
do not intersect the compact ones as indicated by the figure. Moreover, they are also not compatible with a complemented null-space.

\begin{proposition}\label{pro:hyb}
An operator $A: X \to Y$ of hybrid type in the sense of Definition~\ref{def:hybrid} cannot be compact, and its null-space is always uncomplemented. 
\end{proposition}
\begin{proof}
By definition $\Ra(A)$ contains a closed 
infinite-dimensional 
subspace $M$. Take a bounded sequence $(y_n) \subset M$ which does not have a convergent subsequence. 
Such a sequence exists since by the Riesz lemma
since $M$ is infinite dimensional. 
Even though $A$ is non-injective, the 
open mapping theorem still gives a sequence in $x_n \in A^{-1}(y_n)$ with $\|x_n\| \leq C \|y_n\|$ and
$A x_n = y_n$. If $A$ was compact, then $y_n$ would have a convergent subsequence, which is a contradiction. 

If the null-space would be complemented and $A$ is not compact, 
then $X=\mathcal{N}(A) \oplus U$ is connected with an infinite-dimensional closed subspace $U$ of $X$. Then this leads to a contradiction. Namely, let $M \subset \mathcal{R}(A)$ be an 
infinite-dimensional closed subspace of $Y$ and, due to the continuity and injectivity of $A$ on $U$, the set $A_U^{\dagger}[M]$
is an infinite-dimensional closed subspace of $X$ on which $A$ is bounded invertible as a consequence of
Proposition~\ref{pro:Flemming}.
This, however, contradicts the strict singularity of $A$ by Proposition~\ref{rem:Kato}.
\end{proof}

\begin{example}[Hybrid, strictly singular, uncomplemented null-space] \label{ex:hybrid}
\rm 
Examples for strictly singular hybrid operators 
are the Mazur maps: Note that 
for any separable Banach space $Y$ there exists 
a bounded map (Mazur map) $A: X=\ell^1 \to Y$
which is onto (cf.~Example in \cite{GoldThorp63}). Hence the range contains 
a closed infinite-dimensional space (namely $Y$
itself). But by taking $Y = \ell^q$ with $1<q<\infty$,
this is a strictly singular operator
by Proposition~\ref{bpq}, hence 
hybrid. 
Due to Proposition~\ref{pro:hyb}
this operator is non-compact and its null-space is uncomplemented.
We have here ill-posedness of type~II by Proposition~\ref{pro:iff}.
\end{example}

\begin{example}[Non-compact, strictly singular, non-hybrid] \label{ex:embeddings}
\rm
An example of a non-compact but strictly 
singular operator are the embeddings
$\,\mathcal{E}: \ell^p \to \ell^q$ with $1 \le p<q<\infty$. These are injective and
strictly  singular by 
Proposition~\ref{bpq}.  Since $\Nu(A)= \{0\}$
they cannot be hybrid as such instances do 
not exist in the injective case (cf.~Figure~2). 
Moreover this operator is non-compact 
since we may consider the unit-sequence 
$(e_i)_{i=1}^\infty$ with has zero elements except at position 
$i$ where the value is one. Clearly we have 
$\|e_i\|_{\ell^p}=1$ and 
$\|e_i -e_j\|_{\ell^q} = 2^{1/q}$ for $i\not = j$ 
such that the 
bounded sequence $(e_i)_{i=1}^\infty$ cannot have a convergent subsequence. Hence the operator is non-compact, but ill-posed of type~II. 
\end{example}

\begin{example}[General type~I ill-posedness]
\rm
From any ill-posed operator of type~II, we can 
construct  a type~I operator as follows 
by adding the identity: 
Let $A:X \to Y$ be a ill-posed operator. 
Define 
\begin{align*} 
B: \ell^2 \times X &\to \ell^2 \times Y  \\
 ((a_n)_n,x) &\to ((a_n)_n, A x). 
 \end{align*}
We have $\Nu(B) = (0,\Nu(A))$, 
$\Ra(B) = (\ell^2, \Ra(A))$. 
Thus, it is not difficult to see that 
$\Ra(B)$ is closed if and only if $\Ra(A)$ is. 
Moreover $\Nu(B)$ is complemented if and only 
if $\Nu(A)$ is. Thus, $B$ is ill-posed 
by any of the above definitions, if and only 
if $A$ is ill-posed of this type. However since 
$(\ell^2,0)$ constitutes always a closed 
infinite-dimensional subspace in the range, where a bounded inverse 
exists, the operator $B$ is
not strictly singular and thus 
ill-posed of type~I.
\end{example}

\section{Alternative concepts of well-posedness} \label{sec:alternative}
In the following, we consider bounded linear operators $A :X\to Y$ 
between Banach spaces $X$ and $Y$ and discuss further concepts of well-posedness from the literature, in comparison with the already given ones.

\begin{definition}\label{defWP}
We say that the operator equation $A x = y$ (cf.~\eqref{eq:opeq}) is 
\begin{itemize}
\item well-posed in the sense of Hadamard (HaWP) if and only if 
 \[ \Ra(A) = Y \text{ and } \Nu(A) = \{0\};  \]
  \item well-posed in the sense of Hofmann-Scherzer (HSWP) if and only if 
  \[ \Ra(A) \text{ is closed and }  \Nu(A) = \{0\}; \]   
 \item well-posed in the sense of Nashed-Votruba (NVWP) if and only if 
 \begin{align*} &\Ra(A) \text{ closed }, \quad \\
  &\Nu(A) \text{ is complemented in $X$}, \text{ and } \\
 &  \Ra(A) \text{ is complemented in $Y$};
 \end{align*}
 \item  well-posed  (WP) 
 (cf. Def.~\ref{def:general}) if and only if
  \[ \Ra(A) \text{ closed and } 
  \Nu(A) \text{ is complemented in $X$}; 
 \]
 \item Range-well-posed (\RWP) (cf. Def.~\ref{def:type}) 
 if and only if
   \[ \Ra(A) \text{ closed.}
 \]
\end{itemize}
\end{definition}

The definition of (HSWP) was proposed in 
\cite{HofSch98} (see also \cite{Flemmingbuch18}) 
in a slightly different way as the requirement $N(A) = \{0\}$ together with the 
continuity property $A x_n \to A x \Rightarrow x_n \to x$.
A consequence of the open mapping theorem is 
that (HSWP) problems have  a continuous linear 
inverse on its range, such that the latter 
continuity property 
is  equivalent to that listed in Definition~\ref{defWP}. This approach,
which reflects precisely the situation  in Figure~\ref{fig:Banach1},
was motivated by the concept of \emph{local well-posedness} for nonlinear operator equations introduced in \cite[Def.~1.1]{HofSch98}. As also discussed in
\cite{HofPla18}, for linear operator equations, the local character disappears,
because local well-posedness takes place everwhere or nowhere.

The definition of (NVWP) is motivated by the epic study of generalized 
inverses in \cite{Nashedinner87,NashedVotruba76}, where the stated condition 
is equivalent to the existence of a bounded inner inverse mapping $Y \to X$,
where an inner inverse is a linear operator $B$ satisfying $ABA= A$.
Compared to the well-posedness concept in 
Definition~\ref{def:general}, 
the complementedness  of the range is 
 required as additional condition 
 such that (NVWP) serves as 
a stronger condition for well-posedness. 

\begin{example} \rm 
Let $M$ be a closed subspace that is 
not complemented in $X = M \oplus N$. 
Consider the embedding operator $M \to X$.
\begin{align*} A: M &\to X \\
           x &\to x. 
\end{align*}
We have that $\Nu(A) =\{0\}$, $\Ra(A) = M$
is closed. Thus, this is a well-posed problem
by our definition. However, since the range 
is uncomplemented, it is (NV)-ill-posed. 
\end{example}

Note that the open mapping 
theorem shows that 
Hadamard well-posedness (HaWP)
of Definition~\ref{defWP}
also implies stability such 
that we are in accordance 
with the classical 
Hadamard definition. 
Together with our definitions above, 
this constitutes five instances for 
definitions of well-posedness.

Clearly the following implications hold true:  
\[ \text{(HaWP)} 
\quad \Rightarrow  \quad
\begin{array}{c} \text{(NVWP)}  \\
\text{(HSWP)} \end{array}
\quad \Rightarrow  \quad
\text{(WP)}
\quad \Rightarrow \quad 
\text{(\RWP)}. \] 
If $X,Y$ are Hilbert space, then all closed subspaces are complemented and 
  most of the definitions are equivalent: 
\[ 
 \text{(NVWP)}  \\
\quad \Longleftrightarrow  \quad
\text{(WP)}
\quad \Longleftrightarrow  \quad 
\text{(\RWP)} \] 
The same equivalences hold in the Banach space 
case if $A$ is injective 
with dense range.

\begin{remark} \label{rem:quasidistance} \rm
For $y \in \Ra(A)$ and $y_n \in Y$ with $\|y_n-y\|_Y \to 0$ as $n \to \infty$
the  \emph{quasi-distance} 
$${\rm qdist}(A^{-1}[y_n],A^{-1}[y]):= \sup_{\tilde x \in A^{-1}[y_n]}\; \inf_{x \in A^{-1}[y]} \,\|\tilde x-x\|_X,$$
first introduced in \cite{Ivanov63} for general nonlinear problems, is also one possible tool to evaluate stability for linear problems with non-trivial null-spaces $\mathcal{N}(A)$ of the linear operator $A$ mapping between Banach spaces $X$ and $Y$. The monograph \cite{Flemmingbuch18} outlines in Section~1.2 the interplay between range-well-posedness (\RWP) and this kind of stability (labeled in \cite{Flemmingbuch18} as well-posedness in the sense of Ivanov) in the case of multiple solutions (and in a nonlinear setup). Propositions~1.10 and 1.12 in \cite{Flemmingbuch18} show that (\RWP) implies ${\rm qdist}(A^{-1}[y_n],A^{-1}[y]) \to 0$ as $n \to \infty$. If (\RWP)
fails and the null-space is complemented, then this kind of stability also fails.
When $A$ is a linear operator between Banach space  as in this 
article, we may
observe that ${\rm qdist}$
gives  the norm
$\|A^{-1}[y_n] - A^{-1}[y]\|_{X/\Nu(A)}$
in the quotient space $X/\Nu(A)$;
thus this stability requirement is equivalent to 
stability in the quotient space which 
 is equivalent to (\RWP); see the next 
 Theorem~\ref{th:char}.
This  has already been observed in \cite{HofPla18}
for the Hilbert space case. 
\end{remark}

\subsection{Characterizations of well-posedness by
generalized inverses}\label{sec:sub1}
Next, we characterize the well-posedness in Definition~\ref{defWP} and the previously given ones
in terms of generalized inverses, 
building on the seminal work of 
Nashed and Votruba \cite{NashedVotruba76}.
We denote by $L(X,Y)$ the bounded maps from $X \to Y$ and consider only bounded operators $A:X\to Y$.
\begin{theorem}\label{th:char}
A problem $A x = y$ is 
\begin{enumerate}
\item (HaWP) if and only if 
\[ \exists B \in L(Y,X) \text{ with }  B A = Id_X, \quad A B = Id_Y.
 \]
 \item (HSWP) if and only if 
\[ \exists B \in L(\Ra(A),X) \text{ with }  B A = Id_X.
 \]
 \item (NVWP) if and only if 
 \[ \exists B \in L(Y,X)  \text{ with } A B A = A. \] 
  \item (WP) if and only if 
\[ \exists B \in L(\Ra(A),X)  \text{ with } A B A = A .\] 
 \item (\RWP) if and only if 
 \[ \exists \text{(possibly nonlinear) continuous map } B: \overline{\Ra(A)} \to X 
    \text{ with } A B \circ A = A.  
 \]
 or, equivalently, 
 \[ \tilde{A} : X/\Nu(A) \to \Ra(A) \text{ has a bounded linear inverse.}  \]
\end{enumerate} 
\end{theorem}
In the last item we understand by $\tilde{A}$ the operator 
defined on the quotient space $\tilde{A} \eqcl{x} = A x$, with $\eqcl{x}$ the 
equivalence class in $X/\Nu(A)$,
 which is clearly  a well-defined operator.  

\begin{proof}
1.~is classical 
with $B = A^{-1}$. Boundedness of $B$ is 
given by the open mapping theorem. 

Consider 2.: If (HSWP) holds, then $A : X \to \Ra(A)$ is injective and surjective, and 
by the open mapping theorem, a 
continuous right inverse $B: \Ra(A) \to X$ with $B A x = x$ exists. Conversely, if 
such a map exists, then $\Nu(A) = \{0\}$ and convergence of a 
sequence $A x_n \to y$ induces convergence of $x_n = B A x_n \to z$. Hence $y = A z$, and 
the range is closed. 

The equivalence of 3. is proven by Nashed specializing the work 
of Nashed and Votruba \cite{NashedVotruba76}; see \cite[Theorem 7.1]{Nashedinner87}. 

4. If (WP) holds, then by considering $A$ mapping to
the Banach space $\Ra(A)$, we may again use 
 \cite[Theorem 7.1]{Nashedinner87} to show that $B$ exists and is bounded. 
Conversely, if such a $B$ exists, then by the same theorem $\Nu(A)$ must be 
complemented. We show that the range of $A$ is closed: 
We may extend $B$ by continuity to $\overline{\Ra(A)}$. 
Let a sequence satisfy $A x_n \to y$. Then, by continuity 
$B A x_n \to B y = : w$ and  by definition of $B$ as inner inverse $A B A x_n \to y$. 
Hence by continuity of $A$ we have $y = A w$, 
which implies that the range of $A$ is closed.  

Concerning the second conclusion in 5. This is simply the open mapping theorem 
applied to the injective operator $\tilde{A}$. That the range of $A$ is closed 
if $\tilde{A}$ has a bounded inverse follows easily since 
convergence of $A x_n \to y$, implies by  
$A x_n = \tilde{A} \eqcl{x_n}$ convergence of $\tilde{A} \eqcl{x_n}$ and 
thus convergence of $\eqcl{x_n}$. 
Consequently, $y$ is in $\Ra(A) = \Ra(\tilde{A})$. 

Regarding the first equivalence in 5. Assume that such a $B$ exists. Then 
as in the proof of 4. we can conclude that $\Ra(A)$ is closed since linearity 
is not needed there. The opposite direction requires a corollary of the 
Bartle-Graves theorem, Theorem~\ref{BG2}, stated below: Assume that $\Ra(A)$ is closed.  Then $\Ra(A)$ is 
a Banach space and 
$A: X \to \Ra(A)$ is surjective. Hence a continuous right-inverse exists:
$A B(y) = y$ for all $y \in \Ra(A)$. But this is exactly the definition of an
inner inverse: $A B \circ A  = A$.  
\end{proof}

The powerful Bartle--Graves corollary 
that we used in the last part 
reads as follows (cf.~\cite[Corollary 5G.4, p.~298]{DoRo}):
\begin{theorem}\label{BG2}
Let $A:X\to Y$ be a bounded surjective map between Banach spaces $X$ and $Y$. 
Then there is a (in general nonlinear) continuous 
right-inverse $B$, i.e., we have 
\[ A B(y) = y \qquad (\,\forall y \in Y). \] 
\end{theorem}

From an operator point of view, 
the difference between the well-posedness 
and $\Ra$-well-posedness definitions is essentially 
that of whether the problem allows for a 
linear or only a nonlinear continuous solution mapping on its range. The difference to the (NV) concepts 
is that of attainability; in (WP) we only 
consider $y \in \Ra(A)$ as possible data, while 
(NVWP) allows for any $y \in Y$ as data.

As an illustration we may characterize the different well-posedness concepts 
in the following table:
\begin{center}
\begin{tabular}{c|c|c} 
Typ of WP & Range characterization & Operator characterization \\ \hline \hline 
(HaWP) & $\begin{array}{c} \Ra(A) = Y \\ \Nu(A) =
\{0\} \end{array} $ 
& \rule{0mm}{7.8mm}
$\begin{array}{c} 
 \exists  B \in L(Y,X) \\ 
 B A = Id_X \\
 A B = Id_Y 
\end{array}$ \\[2mm] \hline 
(HSWP) & $\begin{array}{c} \Ra(A) \text{ closed}  \\ \Nu(A) = \{0\} \end{array} $
& \rule{0mm}{6.8mm}
$\begin{array}{c} 
 \exists  B \in L(\Ra(A),X) \\ 
 B A = Id_X \\
\end{array}$  \\[2mm]  \hline
(NVWP) & $\begin{array}{c} \Ra(A) \text{ closed}  \\ \Nu(A) \text{ complemented } \\
\Ra(A) 
\text{ complemented }
\end{array} $ & \rule{0mm}{6.8mm}
$\begin{array}{c} 
 \exists B \in L(Y,X) \\ 
 A B A = A \\
\end{array}$  \\[2mm]  \hline
$\begin{array}{c} \text{(WP)} \\
\text{(Def~\ref{def:general})}
\end{array}$
& $\begin{array}{c} \Ra(A) \text{ closed}  \\ \Nu(A) \text{ complemented} 
\end{array} $ & \rule{0mm}{6.8mm}
$\begin{array}{c} 
  \exists  B \in L(\Ra(A),X) \\ 
 A B A = A \\
\end{array}$  \\[2mm]  \hline
$\begin{array}{c} 
\text{\RWP}  \\
\text{(Def~\ref{def:type})}
\end{array}$
& $\begin{array}{c} \Ra(A) \text{ closed} 
\end{array} $ & \rule{0mm}{7.8mm}
$\begin{array}{c} 
  \exists  B: C(\Ra(A),Y) \\
 B \text{ possibly nonlinear }  \\ 
 A B A = A \\
\end{array}$  
\end{tabular}
\end{center}

\subsection{Alternatives concepts of  types of ill-posedness} \label{sec:sub2}
The intuition behind  ``type~I ill-posedness" is 
that an operator contains an 
infinite-dimensional 
{\em well-posed subproblem.} 
As we have seen,  there are various concepts of 
well-posedness,  hence this analogously leads  to 
different concepts of ``type~I ill-posedness", 
depending on what kind of well-posedness we consider 
for the subproblem. Thus, we may introduce a
third alternative type distinction (besides 
those in Definitions~\ref{def:general} and~\ref{def:type})
based on the (NVWP) concept of well-posedness:

\begin{definition}
An  ill-posed  operator $A:X \to Y$ 
is (NV)-type~I ill-posed whenever there exists a 
closed infinite-dimensional subspace $M \subset \Ra(A)$, 
such that 
\[ \Nu(A) \text{ is complemented in $A^{-1}[M]$} \]
and 
\[ M \text{ is complemented in $Y$, }\]
and we call it (NV)-type~II otherwise. 
\end{definition}
Clearly,  (NV)-type I ill-posedness implies our 
type~I ill-posedness definition from Definition~\ref{def:general},
since the complemented range of $M$ is required additionally. Thus, we have the following 
implication within the ill-posed operators: 
\[ \text{(NV)-type I} \subset \text{type I} \subset 
 \text{$\Ra$-type I}.
\]
In particular, 
the class (NV)-type~I ill-posedness does 
not contain all ill-posed operators that are not strictly singular,
but leaves some out. 
Such  operators that are (NV)-type~II ill-posed but of type I 
(not strictly singular) 
(i.e. the set ``$\text{type~I} \setminus \text{(NV)-type~I}$") 
we label as (NV)-hybrid class of operators:

\begin{definition}
 The class of ill-posed operators that 
 are (NV)-type II  but type I 
(non-strictly singular) we call (NV)-hybrid 
operators. 

Equivalently, 
these operators are characterized by the 
fact that there is an infinite-dimensional closed subspace 
$X_1$ such that $A|_{X_1}$ has a bounded inverse on its range but 
where $\Ra(A|_{X_1})$ is not complemented in $Y$. 
Moreover for all infinite-dimensional subspaces spaces $X_1$ which allow for such a bounded 
inverse of $A|_{X_1}$, the subspace  
$\Ra(A|_{X_1})$ is not complemented in $Y$.
\end{definition}

The full picture is now as in Figure~\ref{fig:FHV25}:
\begin{figure}[ht]
\begin{center}
           \begin{tikzpicture}
            \draw[fill={rgb:black,0.5;white,10}, thick]  plot[smooth, tension=.7, fill] coordinates {(-3.5,0.5) (-3,2.5) (-1,3.5) (1.5,3.5)
            (4,3.5) (6,2.5) (6.,-0.1) (4.,-2) (0,-1.9) (-3,-2) (-3.5,0.5)} ;  
                \draw[fill={rgb:black,2.5;white,10}, thick]  plot[smooth, tension=.7, fill] coordinates {(1.5,3.5)
            (4,3.5) (6.5,2.5) (6.5,0.0) (4.5,-2) (0.8,-1.95)    (1.1,0.)  (1.75,1.7)  (1.5,3.5)} ;

            \begin{scope}[shift={(-1,1.5)}]
            \draw[fill={rgb:black,4;white,10},thick]  plot[smooth, tension=.7, fill] coordinates {(2,-0.7) (1.5, 0) (2,1) (3 ,1.3)
            (4 ,1.4)       (5 ,1.4)       (5.5 ,1.0)   (5.5 ,0.4)     (5.3,-0.1)     (4.4,-0.5)     (2.9,-0.75)     (2,-0.7)    
            } ;
            \end{scope}
            
                    \draw[fill={rgb:black,2.5;white,10}, thick]  plot[smooth, tension=.7, fill] coordinates
                    {
               (0.8,-1.75)  (0.55,-1.5)  (0.7,-1.)   (1.15,-0.3)    (2.,-0.5)    (2.,-1.5)    (0.8,-1.75)  
                    };
                    \draw[fill={rgb:black,0.5;white,10}, thick]  plot[smooth, tension=.7, fill] coordinates
                    {
               (0.6,-1.75)  
               (0.55,-1.5) 
                (0.7,-1.)  
                 (0.85,-0.3) 
                    (-0.4,-0.5)  
                      (-0.,-1.5)   
                       (0.6,-1.75)  
                    };
            \draw[thick]
             (1.5,3.5) .. controls (1.8,1.8) and (2.1,2) .. (1.25,0.45) ;
             
            \draw[thick]
             (-3.4,1.5) .. controls (-2.2,1.3) and (-2.0,1.7) .. (1.25,0.45) ;             

 \draw[->,thick]        (3.3,-3)   -- (1.4,-1.55);
 \draw[->,thick]        (1.5,4.5)   -- (1.5,1.7); 
  \node at (1.5,5.2) {\begin{tabular}{c} well posed strictly singular\\ = well-posed compact \\    
                          = finite-dimensional range
                       \end{tabular}
};
 \node at (3.3,-3.5) {{\begin{tabular}{c} strictly singular\\ with $\Ra(A)$ containing\\ closed infinite-dimensional subspace\end{tabular}}};

 \node at (-1.3,-3.5) {{\begin{tabular}{c} 
 Inverse of $A|_{X_{1}}$ exists\\
 (= not strictly singular), \\
 but $\Ra(A|_{X_{1}})$\\
 uncomplemented
 \end{tabular}}};

 \draw[->,thick]        (-1.3,-2.5)   -- (0.3,-1.35);

  \node at (1.4,-1.1) {{\small hybrid}};
    \node at (0.2,-0.7) {{\small (NV)}};
    \node at (0.2,-1.) {{\small hybrid}};
\node at (-0.5,2.1) {\textbf{well-posed}};
\node at (-1.5,0.1) {\textbf{ill-posed of type I}};
\node at (3.9,-.5) {\textbf{ill-posed of type II}};
\node at (-0.2,1.1) [rotate=-15] {operator not};
\node at (-0.8,0.9) [rotate=-15] {strictly singular};
\node at (4.7,0.7) {operator strictly};
\node at (4.9,0.4) {singular};
\node at (2.2,1.7) [rotate=10]  {operator compact};
\end{tikzpicture}        
\caption{Position of hybrid operators.}\label{fig:FHV25}
\end{center}
\end{figure}
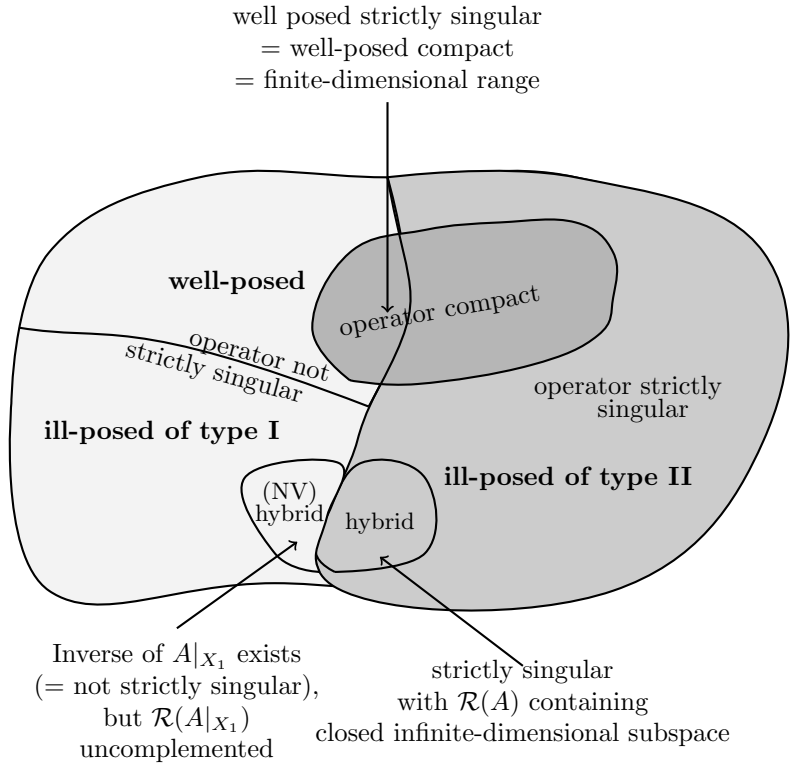
Depending on which choice of type-distinction one 
chooses,  the hybrid cases can be classfied 
differently as type I or type II (or well-posed).

\begin{remark}\rm 
Let us emphasize 
that our terminology of well-/ill-posedness 
and type I/type II is idiosyncratic as there does not 
seem to exist  universally accepted definitions
in the literature. In fact, various authors simply call ``ill-posed" any  of the above ill-posedness concepts without  
a further identifier. 

Since the problem distinction by types is based on a well-posedness concept of a subspace, 
the same holds true for this: depending on what one 
calls well-posed, concrete problems might be attributed 
to different classes. 
Our definition of type~I (cf.~Definition~\ref{def:general}) above is that of Nashed in \cite[Theorem 4.5]{Nashed86}, and it
is also used by Flemming in \cite{Flemmingbuch18}. 
 However, let us mention that Nashed 
 (e.g.~in \cite[Theorem 4.4]{Nashed86}) 
 (preferably) uses the ``(NV)-type I" version. 
On the other hand, in the same article, 
Nashed apparently uses a well-posedness concept (although never explicitly mentioned) based on $\Ra$-well-posedness, which in this combination seems to be slightly inconsequent. 
A similar setup ($\Ra$-well-posedness but 
type~I from Definition~\ref{def:general}) 
is Flemming's  well-posedness concepts in \cite{Flemmingbuch18}. 

In the 
non-injective situations we have three definitions 
of well-posedness (WP, NVWP, $\Ra$-WP), and each of 
this definitions could be paired with a type-distinction 
using one of these definitions
(a la ``contains a (WP, NVWP, $\Ra$-WP) infinite-dimensional subproblem"), 
which would give us a  
3x3 matrix of possible classification 
schemes in the Banach space case, 
where in each situation the borders between types 
and well-posedness in a picture like Figure~\ref{fig:FHV25}
would be different (as for our hybrid cases).

We decided for 
the concept of (WP) from Definition~\ref{def:general}
because of the characterization of both, on the one hand the well-posedness by the existence of a bounded generalized pseudoinverse, and on the other hand  
the equivalence of ill-posed strict singular cases 
with type II problems. Moreover note that our  
type distinction is consistent 
with the well-posedness definition 
and does not 
mix different well-posedness concepts. 
\end{remark}

\begin{remark} \rm
The definition of (NV)-hybrid  operators 
is motivated by Nashed's article, where 
mainly the (NV)-type I  class was considered.  Although the definition 
in \cite[Theorem 4.5]{Nashed86} seems to 
suggest that type I and (NV)-type I  are equivalent, 
this is not the case since in this theorem
the complemented condition on $M$ was forgotten
(but implicitly used) and should be amended there. 
In his paper \cite{Nashed86}, Nashed also states 
that 
``type I is 
 not identical with non-strictly singular", which 
 clearly indicates that he is using the 
 (NV)-type I  definition, and the class of operators 
 which are non-strictly singular but Nashed's type 
 I are precisely our (NV)-hybrid  class. 
\end{remark}

Nashed also gave an example of operators 
that are never (NV)-type I (the proof follows 
\cite[p.~69f]{Nashed86}). 
Compare the next proposition with 
Prop.~\ref{th:pitt}, according to which 
any operator $c_0 \to \ell^p$, $p \in [1,\infty)$,
is never type I: 

\begin{proposition}[Nashed, Schock]\label{propNas}
A bounded operator  $A: c_0 \to \ell^\infty$ is  never (NV)-type I.
\end{proposition}
\begin{proof}
Let $M$ be an infinite-dimensional subspace of $\Ra(A)$ which is 
complemented in $\ell^\infty$ and where $X_1$ defined 
as $M = A X_1$ has a complement in $A^{-1}[M]$. If such an 
$M$ exists, then $A$ would be (NV)-type I, and we show that $M$
cannot exists. 
Let $\tilde{X}_1$ be the complement of $\Nu(A)$ in $X_1$ such that 
$M = \Ra(A|_{\tilde{X}_1})$ and $A: \tilde{X_1} \to M$ is injective (and surjective).
Thus, $A|_{\tilde{X}_1}$ has an bounded inverse $S:= (A|_{\tilde{X}_1})^{-1}:M\to \tilde{X_1}$ , 
and  $\tilde{X_1}$ must be infinite dimensional (since $M$ is so).  Now since $M$ is complemented, there 
exists a continuous projection $Q: \ell^\infty \to M$. The mapping $S Q A: c_0 \to \tilde{X_1} \subset c_0$ 
is the identity on $\tilde{X_1}$ and hence a continuous projection $c_0 \to \tilde{X_1}$. This means 
that $\tilde{X_1}$ is complemented in $c_0$. It is known \cite[Theorem~2.a.3]{LiTz} that then $\tilde{X_1}$ is isomorphic to 
$c_0$; denote the isomorphism by $i:\tilde{X_1} \to c_0$. It follows that 
$i S Q A c_0 \to c_0$ is surjective since $S Q A$ maps surjective to $\tilde{X_1}$. 
It is known \cite[p.~114]{Die} that any bounded map $\ell^\infty \to c_0$ is weakly compact  and weakly compact operators 
form an ideal, hence $S Q A$ is weakly compact.  
Thus, by the isomorphism $i$, 
any bounded set in $c_0$ is weakly compact. 
By the Eberlein-Smulian theorem \cite[p.~18]{Die}, this is equivalent to $c_0$ being reflexive, which is a contradiction. 
\end{proof}
Thus, such mappings from $c_0 \to \ell^\infty$
are either well-posed or (NV)-type~II. 
For instance the embedding $\mathcal{E}:c_0 \to \ell^\infty$
is injective with closed range $c_0$ and thus 
well-posed by our definition;  
it follows from the proof 
of Proposition~\ref{propNas} that it is (NV)-ill-posed
as $c_0$ cannot be complemented in $\ell^\infty$. 

A similar reasoning allows us to construct an 
(NV)-hybrid operator: 
\begin{example}[(NV)-hybrid] \rm 
Consider the following operator: 
\begin{align*}
 B : c_0 \times \ell^2 &\to \ell^\infty \times \ell^2 \\ 
      (x,y) &\to  (x,A y), 
\end{align*}
where $A:\ell^2 \to \ell^2$ is a compact injective operator. Thus, $B$ is the Cartesian product of 
the embedding operator $\mathcal{E}: c_0 \to \ell^\infty$ and a compact operator $A$ mapping in the Hilbert space $\ell^2$. 
We have that $\Nu(B) = \{0\}$ and $\Ra(B) = (c_0,\Ra(A))$. Since $c_0$ is closed, 
$\Ra(B)$ is closed if and only if $\Ra(A)$ is closed, which is not the case, hence 
$B$ constitutes an ill-posed operator 
by our Definition~\ref{def:general}. 
Since the subspace $(c_0,0)$ allows for 
a continuous inverse on this set, the operator is not strictly singular, thus 
of type~I. We shows that it is  not (NV)-type I: Consider an infinite-dimensional closed 
subspace $M = B(X_1)$,
$X_1 \subset c_0 \times \ell^2$ 
such that $M$ is complemented.
We have $M = (M_1,M_2)$ is closed.   Let $N = (N_1,N_2)$ be the 
topological complement of $M$ in $\ell^\infty \times \ell^2$. Then $N$ is closed and 
this implies that $N_1$ and $N_2$ are closed. Thus, $N_1$ constitutes a 
topological complement of $M_1$ in $\ell^\infty$. We also have to 
show that $M_1$ is infinite dimensional. Since $M$ is so, $M_1$ is 
finite dimensional only if $M_2$ is infinite dimensional. However, 
$M_2$ is a closed subspace in the range of the compact operator $A$
and hence can be at most finite dimensional.
Thus, $M_1$ must be 
infinite dimensional. 

As above this now allows for a continuous 
projection $Q:\ell^\infty \to M_1$, and combining this with the inverse of 
the first component of $B|_{(M_1,0)} \to (M_1,0)$, we arrive at a
weakly compact map to $M_1 \subset c_0$ where $M_1$ is complemented in $c_0$
by the same reasoning as in Proposition~\ref{propNas}.
This again yields the contradiction that $c_0$ is reflexive. Thus, no such subset $X_1$ can 
exist, which means that $B$  cannot be 
of (NV)-type I.  
\end{example}


\begin{thebibliography}{10}

\bibitem{Brezis}
H.~Brezis.
\newblock {\em Functional analysis, {S}obolev spaces and partial differential
  equations}.
\newblock Universitext. Springer, New York, 2011.

\bibitem{Die}
J.~Diestel.
\newblock {\em Sequences and Series in {B}anach Spaces}, volume~92 of {\em
  Graduate Texts in Mathematics}.
\newblock Springer-Verlag, New York, 1984.

\bibitem{DoRo}
A.~L. Dontchev and R.~T. Rockafellar.
\newblock {\em Implicit Functions and Solution Mappings}.
\newblock Springer Series in Operations Research and Financial Engineering.
  Springer, New York, second edition, 2014.
\newblock A view from variational analysis.

\bibitem{Flemmingbuch18}
J.~Flemming.
\newblock {\em Variational Source Conditions, Quadratic Inverse Problems,
  Sparsity Promoting Regularization}.
\newblock Frontiers in Mathematics. Birkh\"{a}user/Springer, Cham, 2018.
\newblock New results in modern theory of inverse problems and an application
  in laser optics.

\bibitem{FHV15}
J.~Flemming, B.~Hofmann, and I.~Veseli\'{c}.
\newblock On {$\ell^1$}-regularization in light of {N}ashed's ill-posedness
  concept.
\newblock {\em Comput. Methods Appl. Math.}, 15(3):279--289, 2015.

\bibitem{Gerth21}
D.~Gerth, B.~Hofmann, C.~Hofmann, and S.~Kindermann.
\newblock The {H}ausdorff moment problem in the light of ill-posedness of
  type~{I}.
\newblock {\em Eurasian Journal of Mathematical and Computer Applications},
  9(2):57--87, 2021.

\bibitem{GoldThorp63}
S.~Goldberg and E~Thorp.
\newblock On some open questions concerning strictly singular operators.
\newblock {\em Proc.~Amer.~Math.~Soc.}, 14:334--336, 1963.

\bibitem{Hadamard23}
J.~Hadamard.
\newblock {\em Lectures on the Cauch Problem in Linear Partial Differential
  Equations}.
\newblock Yale University Press, New Haven, 1923.

\bibitem{HofPla18}
B.~Hofmann and R.~Plato.
\newblock On ill-posedness concepts, stable solvability and saturation.
\newblock {\em J. Inverse Ill-Posed Probl.}, 26(2):287--297, 2018.

\bibitem{HofSch98}
B.~Hofmann and O.~Scherzer.
\newblock Local ill-posedness and source conditions of operator equations in
  {H}ilbert spaces.
\newblock {\em Inverse Problems}, 14(5):1189--1206, 1998.

\bibitem{Ivanov63}
V.~K. Ivanov.
\newblock On ill-posed problems (russian).
\newblock {\em Mat. Sbornik (New Series)}, 61(103):211--223, 1963.

\bibitem{Kato58}
T.~Kato.
\newblock Perturbation theory for nullity, deficiency and other quantities of
  linear operators.
\newblock {\em J. Analyse Math.}, 6:261--322, 1958.

\bibitem{KindHof24}
S.~Kindermann and B.~Hofmann.
\newblock Curious ill-posedness phenomena in the composition of non-compact
  linear operators in {H}ilbert spaces.
\newblock {\em J. Inverse Ill-Posed Probl.}, 32(5):1001--1013, 2024.

\bibitem{LiTz}
J.~Lindenstrauss and L.~Tzafriri.
\newblock {\em Classical {B}anach {S}paces. {I}}, volume~92 of {\em Ergebnisse
  der Mathematik und ihrer Grenzgebiete [Results in Mathematics and Related
  Areas]}.
\newblock Springer-Verlag, Berlin-New York, 1977.
\newblock {S}equence {S}paces.

\bibitem{nmh22}
M.~T. Nair, P.~Mathé, and B.~Hofmann.
\newblock Regularization of linear ill-posed problems involving multiplication
  operators.
\newblock {\em Applicable Analysis}, 101(2):714–732, 2022.

\bibitem{Nashedinner87}
M.~Z. Nashed.
\newblock Inner, outer, and generalized inverses in {B}anach and {H}ilbert
  spaces.
\newblock {\em Numer. Funct. Anal. Optim.}, 9(3-4):261--325, 1987.

\bibitem{Nashed86}
M.~Z. Nashed.
\newblock A new approach to classification and regularization of ill-posed
  operator equations.
\newblock In {\em {I}nverse and {I}ll-posed {P}roblems ({S}ankt {W}olfgang,
  1986), volume~4 of Notes Rep.~Math.~Sci.~Engrg.}, pages 53--75. Academic
  Press, Boston, MA, 1987.

\bibitem{NashedVotruba76}
M.~Z. Nashed and G.~F. Votruba.
\newblock A unified operator theory of generalized inverses.
\newblock In {\em Generalized Inverses and Applications (Proceedings of an
  Advanced Seminar Sponsored by the Mathematics Research Center, the University
  of Wisconsin–Madison, October 8–10, 1973)}, pages 1--109. Academic Press,
  New York, 1976.

\bibitem{WernerHof25}
F.~Werner and B.~Hofmann.
\newblock A unified concept of the degree of ill-posedness for compact and
  non-compact linear operator equations in {H}ilbert spaces under the auspices
  of the spectral theorem.
\newblock {\em Numer. Funct. Anal. Optim.}, 46(4-5):322--347, 2025.

\end{thebibliography}

\end{document}